\documentclass[leqno, 12pt]{article}
\usepackage{amsmath,amsfonts,amsthm,amssymb,indentfirst}

\newtheorem{theorem}{Theorem}
\newtheorem{lemma}[theorem]{Lemma}
\newtheorem{definition}[theorem]{Definition}
\newtheorem{corollary}[theorem]{Corollary}
\newtheorem{proposition}[theorem]{Proposition}

\theoremstyle{definition}

\newcommand{\Sym}{\mathrm{Sym}}
\newcommand{\Inj}{\mathrm{Inj}}
\newcommand{\Fin}{\mathrm{Fin}}
\newcommand{\Alt}{\mathrm{Alt}}
\newcommand{\Gen}{\mathrm{Gn}}

\newcommand{\C}{\mathrm{C}}

\newcommand{\Z}{\mathbb{Z}}
\newcommand{\N}{\mathbb{N}}

\begin{document}

\title{Monoids of injective maps closed under conjugation by permutations}

\author{Zachary Mesyan\thanks{This work was done while the author was supported by a Postdoctoral Fellowship from the Center for Advanced Studies in Mathematics at Ben Gurion University, a Vatat Fellowship from the Israeli Council for Higher Education, and ISF grant 888/07.}}

\maketitle

\begin{abstract}
Let $\Omega$ be a countably infinite set, $\Inj(\Omega)$ the
monoid of all injective endomaps of $\Omega$, and $\Sym(\Omega)$
the group of all permutations of $\Omega.$ We classify all
submonoids of $\Inj(\Omega)$ that are closed under conjugation by
elements of $\Sym(\Omega).$
\end{abstract}

\section{Introduction}
Let $\Omega$ be a countably infinite set and $\Sym(\Omega)$ the
group of all permutations of $\Omega.$ In 1933 Schreier and
Ulam~\cite{S&U} showed that $\Sym(\Omega)$ has precisely four
normal subgroups. This result was then generalized by
Baer~\cite{Baer} to arbitrary sets in place of $\Omega.$ (See
also~\cite{Bertram1} and~\cite{D&G} for other related results.) In
this paper we generalize the Schreier-Ulam Theorem in a different
direction, by classifying all the normal submonoids (i.e., ones that are closed under conjugation by elements of $\Sym(\Omega)$) of $\Inj(\Omega),$ the monoid of all injective endomaps of $\Omega.$  Unlike normal subgroups of $\Sym(\Omega),$ there are infinitely many (in fact, $2^{\aleph_0}$) normal submonoids of $\Inj(\Omega)$. However, it is possible to describe them.

Given a normal submonoid $M \subseteq \Inj(\Omega),$ our strategy will be to split $M$ into the smaller semigroups $M_{\mathrm{gp}}$
(consisting of the permutations in $M$), $M_{\mathrm{fin}}$
(consisting of the elements $f \in M$ satisfying $1 \leq |\Omega
\setminus (\Omega)f| < \aleph_0$), and $M_\infty$ (consisting of
the elements $f \in M$ such that $|\Omega \setminus (\Omega)f| =
\aleph_0$). We shall then describe these three subsemigroups (each
of which is also normal) individually. Even though our definition does not explicitly say that $M_{\mathrm{gp}}$ is closed under taking inverses, it turns out that this must be the case, and hence that $M_{\mathrm{gp}}$ must always be one of the four groups mentioned in the Schreier-Ulam Theorem. Further, $M_\infty$ must either be empty or contain every element $f \in \Inj(\Omega)$ satisfying $|\Omega \setminus (\Omega)f| = \aleph_0.$ The semigroup $M_{\mathrm{fin}}$ is more difficult to describe, and its structure depends on that of $M_{\mathrm{gp}}.$ But, roughly speaking, $M_{\mathrm{fin}}$ must either be of the form $$\{f \in \Inj(\Omega) : |\Omega \setminus (\Omega)f| \in N \setminus
\{0\}\},$$ where $N$ is an additive submonoid of the natural numbers, or be a slightly smaller subsemigroup of such a semigroup.

The main tool used in proving the result described above is the
following theorem from~\cite{ZM}: given any three maps $f, g, h
\in \Inj(\Omega)\setminus \Sym(\Omega),$ there exist permutations
$a, b \in \Sym(\Omega)$ such that $h = afa^{-1}bgb^{-1}$ if and
only if $$|\Omega \setminus (\Omega)f| + |\Omega \setminus
(\Omega)g| = |\Omega \setminus (\Omega)h|.$$

\subsection*{Acknowledgements}  

The author is grateful to George Bergman for very helpful comments on an earlier draft of this note, and to the referee for suggesting ways to improve the notation.

\section{Conjugation basics}

We begin with some basic definitions and facts about conjugation of injective set maps. The results in this section are all easy and are discussed in detail in~\cite{ZM}, so we omit their proofs here.

Let $\Omega$ be an arbitrary infinite set, $\Inj(\Omega)$ the
monoid of all injective endomaps of $\Omega$, and $\Sym(\Omega)$
the group of all permutations of $\Omega.$ We shall write set maps
on the right of their arguments. The set of integers will be
denoted by $\Z,$ the set of positive integers will be denoted by
$\Z_+,$ the set of nonnegative integers will be denoted by $\N,$
and the cardinality of a set $\Sigma$ will be denoted by
$|\Sigma|$.

\begin{definition}
Let $f \in \Inj(\Omega)$ be any element, and let $\, \Sigma \subseteq
\Omega$ be a nonempty subset. We shall say that $\, \Sigma$ is a {\em
cycle} under $f$ if the following two conditions are satisfied:
\begin{enumerate}
\item[$(${\rm i}$)$] for all $\alpha \in \Omega,$ $(\alpha)f
\in \Sigma$ if and only if $\alpha \in \Sigma;$
\item[$(${\rm ii}$)$] $\Sigma$ has no proper nonempty subset satisfying
 $(${\rm i}$)$.
\end{enumerate}
We shall say that $\, \Sigma$ is a {\em forward cycle} under $f$
if $\, \Sigma$ is an infinite cycle under $f$ and $\Sigma \setminus (\Omega)f \neq \emptyset.$ If $\, \Sigma$ is an infinite cycle under
$f$ that is not a forward cycle, we shall refer to it as an {\em
open cycle}.
\end{definition}

It is easy to see that for any $\alpha \in \Omega,$ the set
$$\{(\alpha)f^n : n \in \N\}\cup \{\beta \in \Omega : \exists n \in \Z_+ \,
((\beta)f^n = \alpha)\}$$ is a cycle under $f.$ By condition (ii)
above, it follows that every cycle of $f$ is of this form. This
also implies that that every $\alpha \in \Omega$ falls into
exactly one cycle under $f.$ Thus, we can define a collection
$\{\Sigma_i\}_{i\in I}$ of disjoint subsets of $\Omega$ to be a
{\em cycle decomposition} of $f$ if each $\Sigma_i$ is a cycle
under $f$ and $\bigcup_{i\in I} \Sigma_i = \Omega.$ We note that
$f$ can have only one cycle decomposition, up to reindexing the
cycles. For convenience, we shall therefore at times refer to {\em
the} cycle decomposition of $f.$

\begin{definition}
Let $f,g \in \Inj(\Omega)$ be any two elements. We shall say that
$f$ and $g$ have {\em equivalent} cycle decompositions if there
exist an indexing set $I$ and cycle decompositions $\,
\{\Sigma_i\}_{i\in I}$ and $\, \{\Gamma_i\}_{i\in I}$ of $f$ and
$g,$ respectively, that satisfy the following two conditions:
\begin{enumerate}
\item[$(${\rm i}$)$] for all $i \in I,$ $|\Sigma_i| = |\Gamma_i|;$
\item[$(${\rm ii}$)$] if $\, |\Sigma_i| = |\Gamma_i| = \aleph_0$ for some $i \in I,$ then $\, \Sigma_i$ is a forward cycle under $f$ if and only if $\,
\Gamma_i$ is a forward cycle under $g.$
\end{enumerate}
\end{definition}

As with permutations, we have the following fact.

\begin{proposition}[{\rm \cite[Proposition~3]{ZM}}]\label{decomposition}
Let $f,g \in \Inj(\Omega)$ be any two maps. Then $g = afa^{-1}$
for some $a \in \Sym(\Omega)$ if and only if $f$ and $g$ have
equivalent cycle decompositions.
\end{proposition}

The next two observations will also be useful in the sequel.

\begin{lemma}[{\rm \cite[Lemma~4]{ZM}}]\label{fwd cycles}
Let $f \in \Inj(\Omega)$ be any map. Then there is a one-to-one
correspondence between the elements of $\, \Omega \setminus
(\Omega)f$ and forward cycles in the cycle decomposition of $f$.
\end{lemma}

\begin{lemma}[{\rm \cite[Lemma~5]{ZM}}]\label{coimage}
Let $f,g \in \Inj(\Omega)$ be any two maps. Then $$|\Omega
\setminus (\Omega)f| + |\Omega \setminus (\Omega)g| = |\Omega
\setminus (\Omega)fg|.$$
\end{lemma}

We conclude this section by recalling another generalization of a familiar concept from group theory.

\begin{definition}
We shall say that a subset $M \subseteq \Inj(\Omega)$ is \emph{normal} if it is closed under conjugation by elements of $\, \Sym(\Omega).$
\end{definition}

\section{Some general considerations} \label{gen.sect}

From now on we shall assume that $\Omega$ is countable. Let
$\Inj_{\mathrm{fin}}(\Omega) \subseteq \Inj(\Omega)$ denote the
subset consisting of all elements $f$ such that $1 \leq |\Omega
\setminus (\Omega)f| < \aleph_0,$ and let $\Inj_{\infty}(\Omega)
\subseteq \Inj(\Omega)$ denote the subset consisting of all elements
$f$ such that $|\Omega \setminus (\Omega)f| = \aleph_0.$ By
Lemma~\ref{coimage}, these two sets are subsemigroups of
$\Inj(\Omega).$ Further, by Proposition~\ref{decomposition} and
Lemma~\ref{fwd cycles}, if $f,h \in \Inj(\Omega)$ are conjugate to
each other, then $|\Omega \setminus (\Omega)f| = |\Omega \setminus
(\Omega)h|.$ Hence, $\Inj_{\mathrm{fin}}(\Omega)$ and
$\Inj_{\infty}(\Omega)$ are normal. We shall not delve very deeply into its structure, but let us note that $\Inj_{\infty}(\Omega)$ is known as the \emph{Baer-Levi semigroup} (of type $(\aleph_0,\aleph_0)$). See, for instance, \cite{CP} and \cite{LM} for more information about it.

Given a submonoid $M \subseteq \Inj(\Omega),$ we set
$M_{\mathrm{gp}} = M \cap \Sym(\Omega),$ $M_{\mathrm{fin}} = M
\cap \Inj_{\mathrm{fin}}(\Omega),$ and $M_\infty = M \cap
\Inj_{\infty}(\Omega).$ Then $M = M_{\mathrm{gp}} \cup
M_{\mathrm{fin}} \cup M_\infty,$ and the union is disjoint. As
intersections of semigroups, $M_{\mathrm{gp}},$
$M_{\mathrm{fin}},$ and $M_\infty$ are semigroups.

Suppose that a submonoid $M \subseteq \Inj(\Omega)$ is normal. Then, by the above remarks, the same is true of $M_{\mathrm{gp}},$
$M_{\mathrm{fin}},$ and $M_\infty.$ In order to understand the
structure of $M,$ we shall try to understand the structures of
these three ``pieces" of $M$ individually. In the case of
$M_{\mathrm{gp}},$ we can accomplish this task very quickly, by
relying on the Schreier-Ulam Theorem. We require a little more notation in
order to state it in full detail.

\begin{definition}
Given a map $g \in \Sym(\Omega),$ the {\em support} of $g$ is the set $\, \{\alpha \in \Omega : (\alpha)g \neq \alpha\}.$ The subgroup of $\, \Sym(\Omega)$ consisting of all the elements having finite support will be denoted by $\, \Fin(\Omega).$ The elements of $\, \Fin(\Omega)$ are said to be {\em finitary}. Finally, $\, \Alt(\Omega) \subseteq \Fin(\Omega)$
will denote the alternating subgroup {\rm(}consisting of even
finitary permutations{\rm)}.
\end{definition}

\begin{theorem}[Schreier and Ulam~\cite{S&U}]\label{ulam}
$\Sym(\Omega)$ has precisely four normal subgroups, specifically,
$\{1\},$ $\Alt(\Omega),$ $\Fin(\Omega),$ and $\, \Sym(\Omega).$
\end{theorem}

\begin{lemma}\label{invert.elts.form.gp}
Let $M \subseteq \Inj(\Omega)$ be a normal submonoid. Then
$M_{\mathrm{gp}} = M \cap \Sym(\Omega)$ is a normal subgroup of
$\, \Sym(\Omega).$
\end{lemma}

\begin{proof}
As we have noted above, $M_{\mathrm{gp}}$ must be a normal submonoid of $\Inj(\Omega).$ Hence, it suffices to show that $M_{\mathrm{gp}}$ is closed under taking inverses. But, for any permutation $f \in \Sym(\Omega)$, it is easy to see that $f$ and $f^{-1}$ have equivalent cycle decompositions, and hence  $f^{-1} = afa^{-1}$ for some $a \in \Sym(\Omega)$, by Proposition~\ref{decomposition}. Therefore, since $M_{\mathrm{gp}}$ is closed under conjugation, it is closed under taking inverses as well.
\end{proof}

By the Lemma~\ref{invert.elts.form.gp} and Theorem~\ref{ulam}, if $M \subseteq \Inj(\Omega)$ is a normal submonoid, then $M_{\mathrm{gp}}$ must be one of
$\{1\},$ $\Alt(\Omega),$ $\Fin(\Omega),$ and $\Sym(\Omega).$ 

Let us next recall a result mentioned in the Introduction that will play an important role throughout this note.

\begin{theorem}[{\rm \cite[Corollary~10]{ZM}}]\label{composition}
Let $f, g, h \in \Inj(\Omega)\setminus \Sym(\Omega)$ be any three maps.
Then there exist permutations $a, b \in \Sym(\Omega)$ such that
$h = afa^{-1}bgb^{-1}$ if and only if $$|\Omega \setminus
(\Omega)f| + |\Omega \setminus (\Omega)g| = |\Omega \setminus
(\Omega)h|.$$
\end{theorem}

\begin{corollary}\label{inf.coimage.monoid}
Let $M \subseteq \Inj(\Omega)$ be a normal submonoid. Then either $M \cap \Inj_{\infty}(\Omega) = \emptyset$ or $\, \Inj_{\infty}(\Omega) \subseteq M.$
\end{corollary}

\begin{proof}
Suppose that there is an element $f \in M \cap \Inj_{\infty}(\Omega).$ Then, $afa^{-1}bfb^{-1} \in M$ for all $a, b \in \Sym(\Omega),$ since $M$ is closed under conjugation and composition. Hence, $\Inj_{\infty}(\Omega) \subseteq M,$ by the previous theorem.
\end{proof}

Let $M \subseteq \Inj(\Omega)$ be a normal submonoid. Theorem~\ref{ulam}, Lemma~\ref{invert.elts.form.gp}, and Corollary~\ref{inf.coimage.monoid} alow us to completely characterize $M_{\mathrm{gp}}$ and $M_\infty$, and so it remains to explore $M_{\mathrm{fin}}.$ Unlike $M_\infty,$ the structure of $M_{\mathrm{fin}}$ depends on whether $M_{\mathrm{gp}}$ is $\{1\},$ $\Alt(\Omega),$ $\Fin(\Omega),$ or $\Sym(\Omega).$ We shall discuss these four cases individually in the four sections that follow. Then, in Section~\ref{main.section} we shall collect all those pieces for a complete description of the normal submonoids of $\Inj(\Omega)$. It is easiest to describe $M_{\mathrm{fin}}$ when $M_{\mathrm{gp}} = \Sym(\Omega),$ so we start there.

\section{Monoids containing $\Sym(\Omega)$}

We note that any submonoid of $\Inj(\Omega)$ that contains
$\Sym(\Omega)$ is automatically normal. More generally, we have the following fact.

\begin{lemma}\label{contains sym}
Let $M \subseteq \Inj(\Omega)$ be any submonoid that contains $\,
\Sym(\Omega),$ and let $f \in M$ be any element. Suppose that $g
\in \Inj(\Omega)$ satisfies $\, |\Omega \setminus (\Omega)g| =
|\Omega \setminus (\Omega)f|.$ Then $g \in M.$
\end{lemma}

\begin{proof}
Since $f$ and $g$ are injective, the formula $((\alpha)f)h = (\alpha)g$ defines a bijection $h: (\Omega)f \rightarrow (\Omega)g$. Moreover, since $|\Omega \setminus (\Omega)f| = |\Omega \setminus (\Omega)g|$, we can extend $h$ to a permutation of $\Omega$, using any bijection $\Omega \setminus (\Omega)f \rightarrow \Omega \setminus (\Omega)g$. The desired conclusion now follows from the fact that $g = fh \in M.$
\end{proof}

\begin{definition}
Given a subset $M \subseteq \Inj(\Omega),$ we define $$M_\N = \{|\Omega \setminus (\Omega)f| : f \in M\} \cap \N.$$ Also, given a subset $N \subseteq \N$, we define $$S(N) = \{f \in \Inj(\Omega) : |\Omega \setminus (\Omega)f| \in N \setminus \{0\}\}.$$
\end{definition}

By Lemma~\ref{coimage}, $M_\N$ is a submonoid of the additive monoid $\N$ whenever $M$ is a submonoid of $\Inj(\Omega)$, and $S(N)$  is a subsemigroup of $\Inj(\Omega)$ whenever $N$ is a subsemigroup of $\N$. It is also easy to see that $S(N)$ is normal.

We can now quickly describe the submonoids of $\Inj(\Omega)$ that contain $\Sym(\Omega)$.

\begin{proposition}\label{sym.classif}
For any submonoid $N \subseteq \N$, both $\, \Sym(\Omega) \cup S(N)$ and $\, \Sym(\Omega) \cup S(N) \cup \Inj_{\infty}(\Omega)$ are submonoids of $\, \Inj(\Omega)$.

Conversely, a submonoid $M \subseteq \Inj(\Omega)$ that contains $\, \Sym(\Omega)$ must be either of the form $\, \Sym(\Omega) \cup S(N)$ or of the form $\, \Sym(\Omega) \cup S(N) \cup \Inj_{\infty}(\Omega)$, for some submonoid $N \subseteq \N$. $($Specifically, $N = M_\N$.$)$  
\end{proposition}

\begin{proof}
By Lemma~\ref{coimage}, for any submonoid $N \subseteq \N$, $S(N)$ is closed under multiplication by elements of $\Sym(\Omega)$, and $\Inj_{\infty}(\Omega)$ closed under multiplication by elements of $\Sym(\Omega) \cup S(N)$. The first claim now follows from the fact that $\Sym(\Omega)$, $S(N)$, and $\Inj_{\infty}(\Omega)$ are all subsemigroups of $\Inj(\Omega)$.

For the converse, let $M \subseteq \Inj(\Omega)$ be a submonoid containing $\Sym(\Omega)$. Then, by Lemma~\ref{contains sym}, an element $f \in \Inj_{\mathrm{fin}}(\Omega)$ is in $M_{\mathrm{fin}}$ if and only if $|\Omega \setminus (\Omega)f| \in M_\N \setminus \{0\}.$ Hence,
$M_{\mathrm{fin}} = S(M_\N).$ By the same lemma (or, by
Corollary~\ref{inf.coimage.monoid}), $M_\infty = M \cap \Inj_{\infty}(\Omega)$ must be either empty or all of $\Inj_{\infty}(\Omega).$ Thus, either $M = \Sym(\Omega) \cup S(M_\N)$ or $M = \Sym(\Omega) \cup S(M_\N) \cup \Inj_{\infty}(\Omega).$
\end{proof}

In the following section we shall discuss the next simplest case, of
submonoids $M \subseteq \Inj(\Omega)$ such that $M \cap
\Sym(\Omega) = \{1\}.$

\section{Monoids with trivial groups of units}

We shall require the following well-known observation about
additive submonoids of $\N.$

\begin{lemma}
Every additive submonoid $N \subseteq \N$ has a unique minimal
generating set.
\end{lemma}

\begin{proof}
This is clear if $N = \{0\},$ so we may assume that $N \neq
\{0\}.$ Let $\{G_i\}_{i\in I}$ be the set of all generating sets
for $N$ as a monoid. We shall show that $G = \bigcap_{i\in I} G_i$ is a generating set for $N.$ (Note that $G$ is necessarily nonempty, since it contains the least nonzero element of N.) 

Suppose that $G$ does not generate $N.$ Let $n \in N \setminus
\{0\}$ be the least element that is not in $\langle G \rangle,$
the monoid generated by $G$. Then there must be some generating set $G_i$ ($i\in I$) such that $n \notin G_i.$ Hence $n = n_1 + \dots + n_k$ for some $n_1, \dots, n_k \in G_i,$ since $G_i$ is a generating set for $N.$ We must necessarily have $n_1, \dots, n_k < n.$ By our choice of $n,$
this implies that $n_1, \dots, n_k \in \langle G \rangle.$ Hence
$n \in \langle G \rangle$; a contradiction. Therefore $G$
generates all of $N.$
\end{proof}

We can thus make the following

\begin{definition}
Given an additive submonoid $N \subseteq \N,$ let $\, \Gen(N)$
denote the unique minimal generating set for $N$ as a monoid.
\end{definition}

We note, in passing, that $\Gen(N)$ is always finite (e.g., see
\cite[Theorem 2.4(2)]{Gilmer}). We are now ready to describe the normal submonoids of $\Inj(\Omega)$ that have no nontrivial units.

\begin{proposition}\label{gp=fin}
Let $M \subseteq \Inj(\Omega)$ be any submonoid such that $M \cap
\Sym(\Omega) = \{1\}.$ Then $M$ is normal if and only if either $M = \{1\} \cup M_{\mathrm{fin}}$ or $M = \{1\} \cup M_{\mathrm{fin}} \cup \Inj_{\infty}(\Omega)$, where $$M_{\mathrm{fin}} = B \cup \{f \in \Inj(\Omega) :
|\Omega \setminus (\Omega)f| \in N \setminus (\Gen(N) \cup
\{0\})\},$$ for some additive submonoid $N \subseteq \N$, and normal subset $B \subseteq \Inj_{\mathrm{fin}}(\Omega)$ that satisfies $B_\N = \Gen(N).$
\end{proposition}

\begin{proof}
Suppose that $M$ is normal. Let $N = M_\N,$ and set $$B = \{f \in M : |\Omega \setminus (\Omega)f| \in \Gen(N)\}.$$ Then $B$ is normal, since $M$ is. Further, if $h \in \Inj(\Omega)$ is any element such that $|\Omega \setminus (\Omega)h| \in M_\N \setminus (\Gen(M_\N) \cup \{0\}),$ then $h
\in M,$ by Theorem~\ref{composition}. Hence, $$B \cup \{f \in \Inj(\Omega) : |\Omega \setminus (\Omega)f| \in N \setminus (\Gen(N) \cup
\{0\})\} = M \cap \Inj_{\mathrm{fin}} (\Omega).$$ The desired conclusion then follows from Corollary~\ref{inf.coimage.monoid}.

For the converse, suppose that $M_{\mathrm{fin}}$ has the form specified in the statement. First, we note that this set is a subsemigroup of $\Inj(\Omega)$. For, given any two elements $f, g \in M_{\mathrm{fin}}$, we have $$fg \in \{f \in \Inj(\Omega) : |\Omega \setminus (\Omega)f| \in N \setminus (\Gen(N) \cup \{0\})\},$$ by Lemma~\ref{coimage}. As usual, this implies that $M$ is a submonoid. Since $B$ is normal, so is $M_{\mathrm{fin}}$, considering that $$\{f \in \Inj(\Omega) : |\Omega \setminus (\Omega)f| \in N \setminus (\Gen(N) \cup \{0\})\}$$ is always normal. It follows that $M$ is normal as well.
\end{proof}

In the above statement we describe the structure of a normal submonoid of $\Inj(\Omega)$ in terms of a normal set $B$, which may, at first glance, seem not especially helpful. However, by Proposition~\ref{decomposition}, constructing such a set $B$ simply amounts to picking any subset of $\Inj(\Omega)$ satisfying $B_\N = \Gen(N)$ and then adding to it all maps that have cycle decompositions equivalent to those of the maps already in $B$.

\section{Monoids containing $\Fin(\Omega)$}

Our next goal is to describe the normal submonoids $M \subseteq \Inj(\Omega)$ satisfying $M \cap \Sym(\Omega) = \Fin(\Omega).$ To accomplish this we shall first describe how composition with elements of $\Fin(\Omega)$ affects the cycle decomposition of an arbitrary element of $\Inj(\Omega).$

The following notation will be convenient in the sequel.

\begin{definition}
For each $f \in \Inj(\Omega)$ and $n \in \Z_+$ let $$(f)\C_n =
|\{\Sigma \subseteq \Omega : \Sigma \ \mathrm{is} \ \mathrm{a} \
\mathrm{cycle} \ \mathrm{under} \ f \ \mathrm{of} \
\mathrm{cardinality} \ n\}|.$$ Similarly, let
$$(f)\C_{\mathrm{open}} = |\{\Sigma \subseteq \Omega : \Sigma \
\mathrm{is} \ \mathrm{an} \ \mathrm{open} \ \mathrm{cycle} \
\mathrm{under} \ f\}|$$ and
$$(f)\C_{\mathrm{fwd}} = |\{\Sigma \subseteq \Omega :  \Sigma \ \mathrm{is} \
\mathrm{a} \ \mathrm{forward} \ \mathrm{cycle} \ \mathrm{under} \
f\}|.$$
\end{definition}

By Proposition~\ref{decomposition}, two elements $f,g \in
\Inj(\Omega)$ are conjugates of one another if and only if
$(f)\C_{\mathrm{open}} = (g)\C_{\mathrm{open}},$
$(f)\C_{\mathrm{fwd}} = (g)\C_{\mathrm{fwd}},$ and $(f)\C_n
= (g)\C_n$ for all $n \in \Z_+.$ We shall also require a more general
equivalence relation on elements of $\Inj(\Omega).$

\begin{definition} \label{fin.equiv.def}
Given any two maps $f,g \in \Inj(\Omega),$ let us write $f \approx_{\mathrm{fin}} g$ if the following four conditions are satisfied:
\begin{enumerate}
\item[$(${\rm i}$)$] $(f)\C_{\mathrm{open}} =
(g)\C_{\mathrm{open}};$
\item[$(${\rm ii}$)$] $(f)\C_{\mathrm{fwd}} =
(g)\C_{\mathrm{fwd}};$
\item[$(${\rm iii}$)$]$(f)\C_n \neq (g)\C_n$ for only finitely
many $n \in \Z_+;$
\item[$(${\rm iv}$)$] if $(f)\C_n \neq (g)\C_n$ for some $n \in \Z_+,$
then $(f)\C_n$ and $(g)\C_n$ are both finite.
\end{enumerate}
Further, we shall say that a subset $B \subseteq \Inj(\Omega)$ is $\, \approx_{\mathrm{fin}}$-closed if for all $f, g \in \Inj(\Omega)$ such that $f \approx_{\mathrm{fin}} g$, $f \in B$ if and only if $g \in B.$
\end{definition}

Clearly, $\approx_{\mathrm{fin}}$ is an equivalence relation on $\Inj(\Omega).$ We shall prove that for $f,g \in \Inj(\Omega)$, each having at least one infinite cycle in its cycle decomposition, $f \approx_{\mathrm{fin}} g$ if and only if $f$ is conjugate to $h_1 g h_2$ for some $h_1, h_2 \in \Fin(\Omega).$ The argument is
divided into several steps.

\begin{lemma}\label{fin.equiv.0.5}
Let $f \in \Inj(\Omega)$ and $h \in \Fin(\Omega)$ be any two maps. Then
\begin{enumerate}
\item[$(${\rm i}$)$] $(f)\C_{\mathrm{open}} =
(hf)\C_{\mathrm{open}},$ and
\item[$(${\rm ii}$)$] $(f)\C_{\mathrm{open}} = (fh)\C_{\mathrm{open}}.$
\end{enumerate}
\end{lemma}

\begin{proof}
(i) Since $h$ can be written as a product of transpositions, it is
enough to show this in the case where $h$ is a transposition.
Further, since under this assumption $f = hhf,$ it is enough to
show that $(f)\C_{\mathrm{open}} \leq (hf)\C_{\mathrm{open}}.$
This is clear if $h$ fixes all the elements in the open cycles of $f,$ so suppose that $(\alpha)h = \beta \neq \alpha$ for some $\alpha \in \Sigma,$
where $\Sigma$ is an open cycle of $f.$ We consider several
different cases.

Suppose that $\beta \in \Sigma.$ Without loss of generality, we
may assume that $\beta = (\alpha)f^n$ for some $n \in \Z_+.$ Then
$\Sigma \setminus \{(\alpha)f, (\alpha)f^2, \dots, (\alpha)f^n\}$
is an open cycle of $hf.$ Since $h$ is a transposition, all open
cycles of $f$ other than $\Sigma$ are open cycles of $hf.$ Hence
$(f)\C_{\mathrm{open}} \leq (hf)\C_{\mathrm{open}}.$

Suppose instead that $\beta \notin \Sigma,$ and let $\Gamma$ be
the cycle of $f$ that contains $\beta.$ If $\Gamma$ is finite,
then $\Sigma \cup \Gamma$ is an open cycle under $hf.$ If $\Gamma$
is a forward cycle, then $$\{(\beta)f^n : n \in \Z_+\}\cup
\{\gamma \in \Omega : \exists n \in \N \, ((\gamma)f^n =
\alpha)\}$$ is an open cycle under $hf$, in place of $\Sigma$. If $\Gamma$ is an open cycle of $f,$ then $$\{(\beta)f^n : n \in \Z_+\}\cup \{\gamma \in
\Omega : \exists n \in \N \, ((\gamma)f^n = \alpha)\}$$ and
$$\{(\alpha)f^n : n \in \Z_+\}\cup \{\gamma \in \Omega : \exists n
\in \N \, ((\gamma)f^n = \beta)\}$$ are open cycles of $hf$, in place of $\Sigma$ and $\Gamma$. Again, in each of these three cases, $(f)\C_{\mathrm{open}} \leq (hf)\C_{\mathrm{open}}.$

(ii) By part (i), we have $(f)\C_{\mathrm{open}} =
(hf)\C_{\mathrm{open}}.$ Also, Proposition~\ref{decomposition}
implies that $(hf)\C_{\mathrm{open}} =
(h^{-1}(hf)h)\C_{\mathrm{open}}.$ But, the latter is just
$(fh)\C_{\mathrm{open}},$ which completes the proof.
\end{proof}

\begin{corollary}\label{fin.equiv.1}
Let $f, g \in \Inj(\Omega)$ be any two maps, and suppose that $f = hh_1gh_2h^{-1}$ for some $h \in \Sym(\Omega)$ and $h_1, h_2 \in \Fin(\Omega).$ Then $f \approx_{\mathrm{fin}} g.$
\end{corollary}

\begin{proof}
By Proposition~\ref{decomposition}, $f \approx_{\mathrm{fin}}
h_1gh_2.$ Thus, it suffices to show that $h_1gh_2 \approx_{\mathrm{fin}}
g$.

By the previous lemma, $(g)\C_{\mathrm{open}} =
(h_1g)\C_{\mathrm{open}} = (h_1gh_2)\C_{\mathrm{open}}.$ Since
$|\Omega \setminus (\Omega)g| = |\Omega \setminus
(\Omega)h_1gh_2|,$ Lemma~\ref{fwd cycles} implies that
$(g)\C_{\mathrm{fwd}} = (h_1gh_2)\C_{\mathrm{fwd}}.$ The
desired conclusion then follows from the fact that $g$ and
$h_1gh_2$ can disagree on only finitely many elements of $\Omega$, which implies that these two maps must have the same finite cycles, except for possibly finitely many.
\end{proof}

\begin{lemma}\label{fin.equiv.2}
Let $f \in \Inj(\Omega)$ be a map that has at least one infinite
cycle in its cycle decomposition, and let $n \in \Z_+$. Then there
exists a transposition $h \in \Fin(\Omega)$ such that $\,
(fh)\C_n = (f)\C_n + 1$ and $\, (fh)\C_m = (f)\C_m$ for
all $m \in \Z_+ \setminus \{n\}.$
\end{lemma}

\begin{proof}
Let $\Sigma \subseteq \Omega$ be an infinite cycle of $f.$ Then we
can write $\Sigma = \{\alpha_i : i \in I\},$ where either $I = \Z$
or $I = \Z_+,$ and $(\alpha_i)f = \alpha_{i+1}$ for all $i \in I.$
Let us fix an element $\alpha_i \in \Sigma$ and define $h \in
\Fin(\Omega)$ by $(\alpha_i)h = \alpha_{i+n},$ $(\alpha_{i+n})h =
\alpha_i,$ and $(\alpha)h = \alpha$ for all $\alpha \in \Omega
\setminus \{\alpha_i, \alpha_{i+n}\}.$ Then $\{\alpha_i, \dots,
\alpha_{i+n-1}\}$ and $\Sigma \setminus \{\alpha_i, \dots,
\alpha_{i+n-1}\}$ become cycles under $fh$, in place of $\Sigma$, and otherwise $f$ and $fh$ have the same cycle decomposition.
\end{proof}

\begin{lemma}\label{fin.equiv.3}
Let $f \in \Inj(\Omega)$ be a map that has at least one infinite
cycle in its cycle decomposition, and let $n \in \Z_+$ be such
that $\, (f)\C_n > 0.$ Then there exists a transposition $h \in
\Fin(\Omega)$ such that $\, (fh)\C_n = (f)\C_n - 1$ and $\,
(fh)\C_m = (f)\C_m$ for all $m \in \Z_+ \setminus \{n\}.$
\end{lemma}

\begin{proof}
Let $\Sigma \subseteq \Omega$ be an infinite cycle and $\{\beta_1, \dots, \beta_n\}$ an  $n$-cycle in the cycle decomposition of $f$. Let us fix an element $\alpha \in \Sigma$ and define $h \in \Fin(\Omega)$ by $(\alpha)h = \beta_1,$ $(\beta_1)h = \alpha,$ and $(\gamma)h = \gamma$ for all $\gamma \in \Omega \setminus \{\alpha, \beta_1\}.$ Then $\Sigma \cup \{\beta_1, \dots, \beta_n\}$ becomes a cycle under $fh,$ and otherwise $f$ and $fh$ have the same cycle decomposition. Thus,
$fh$ has one fewer $n$-cycle than $f$ but the same finite cycles of other cardinalities.
\end{proof}

Putting together the last four results, we obtain our description
of $\approx_{\mathrm{fin}}$ in terms of composition with elements
of $\Fin(\Omega)$, for maps having infinite cycles.

\begin{proposition}\label{fin.equiv.main}
Let $f, g \in \Inj(\Omega)$ be any two elements, each having at least one infinite cycle in its cycle decomposition. Then $f \approx_{\mathrm{fin}} g$ if and only if $f = hh_1gh_2h^{-1}$ for some $h \in \Sym(\Omega)$ and $h_1, h_2 \in \Fin(\Omega).$
\end{proposition}

\begin{proof}
By Corollary~\ref{fin.equiv.1}, we only need to show the forward
implication, so let us assume that $f \approx_{\mathrm{fin}} g.$
Repeatedly applying the previous two lemmas, we can find a finite
sequence of transpositions $h_1, \dots, h_n \in \Fin(\Omega)$ such that $gh_1 \dots h_n$ has a cycle decomposition equivalent to that of $f.$ (For any $h_1, \dots, h_n \in \Fin(\Omega)$, we have $(gh_1 \dots h_n)\C_{\mathrm{open}} = (f)\C_{\mathrm{open}}$, by Lemma~\ref{fin.equiv.0.5}, and $(gh_1 \dots h_n)\C_{\mathrm{fwd}} = (f)\C_{\mathrm{fwd}}$, by Lemma~\ref{fwd cycles}.) The result then follows from Proposition~\ref{decomposition}.
\end{proof}

In the above proposition, the assumption that $f$ and $g$ both have an
infinite cycle is necessary. For instance, let $f \in \Sym(\Omega)$ be an element such that $(f)\C_{\mathrm{open}} = 0$ and $(f)\C_n = 1$ for all $n \in \Z_+,$ and let $g \in \Sym(\Omega)$ be an element such that $(g)\C_{\mathrm{open}} = 0,$ $(g)\C_1 = 2,$ and $(g)\C_n = 1$ for all $n >1.$ Then, clearly, $f \approx_{\mathrm{fin}} g.$ But, $f \neq hh_1gh_2h^{-1}$ for all $h \in \Sym(\Omega)$ and $h_1, h_2 \in \Fin(\Omega).$ For, supposing otherwise, there exist $h_1, h_2 \in \Fin(\Omega)$ such that $h_1gh_2$ has a cycle decomposition equivalent to that of $f,$ by Proposition~\ref{decomposition}. Let us list the cycles of $g$ as $\{\alpha_1\}, \{\alpha_2\}, \{\alpha_3, \alpha_4\}, \{\alpha_5, \alpha_6, \alpha_7\}, \{\alpha_8, \alpha_9, \alpha_{10}, \alpha_{11}\}, \dots,$ where
$\Omega = \{\alpha_i : i \in \Z_+\}.$ Since $g$ and $h_1gh_2$ can
disagree on only finitely many elements of $\Omega,$ there is a positive integer $n$ such that for all $i > n$ we have $(\alpha_i)g =
(\alpha_i)h_1gh_2,$ and such that $\alpha_{n+1}$ is the element
with the least index in some cycle of $g.$ Let us denote the
cardinality of the cycle to which $\alpha_{n+1}$ belongs by $m.$
Then, by our definition of $f,$ $\{\alpha_1, \dots, \alpha_n\}$
must contain exactly one cycle of $h_1gh_2$ of each cardinality
less than $m$ and no other cycles. Comparing this with our cycle
decomposition for $g$ yields a contradiction (since $\{\alpha_1,
\dots, \alpha_n\}$ contains two cycles of $g$ of cardinality $1,$
in addition to a cycle of each cardinality less than $m$ but
greater than $1$).

We are now ready to describe the normal submonoids $M \subseteq
\Inj(\Omega)$ having the property that $M \cap \Sym(\Omega) = \Fin(\Omega).$

\begin{proposition}\label{fin.classif}
Let $M \subseteq \Inj(\Omega)$ be any submonoid such that $M \cap
\Sym(\Omega) = \Fin(\Omega).$ Then $M$ is normal if and only if either $M = \Fin(\Omega) \cup M_{\mathrm{fin}}$ or $M = \Fin(\Omega) \cup M_{\mathrm{fin}} \cup \Inj_{\infty}(\Omega)$, where $$M_{\mathrm{fin}} = B \cup \{f \in \Inj(\Omega) : |\Omega \setminus (\Omega)f| \in N \setminus (\Gen(N) \cup \{0\})\},$$ for some additive submonoid $N \subseteq \N$ and some $\, \approx_{\mathrm{fin}}$-closed subset $B \subseteq \Inj_{\mathrm{fin}}(\Omega)$ that satisfies $B_\N = \Gen(N)$.
\end{proposition}

\begin{proof}
This proof is very similar to that of Proposition~\ref{gp=fin}.

Suppose that $M$ is normal. Let $N = M_\N,$ and set $$B = \{f \in M : |\Omega \setminus (\Omega)f| \in \Gen(N)\}.$$ Since $M$ is normal and contains $\Fin(\Omega)$, $B$ is $\approx_{\mathrm{fin}}$-closed, by Proposition~\ref{fin.equiv.main}, and it clearly satisfies $B_\N = \Gen(N)$. Further, if $h \in \Inj(\Omega)$ is any element such that $|\Omega \setminus (\Omega)h| \in M_\N \setminus (\Gen(M_\N) \cup \{0\}),$ then $h \in M,$ by Theorem~\ref{composition}. Hence, $$B \cup \{f \in \Inj(\Omega) : |\Omega \setminus (\Omega)f| \in N \setminus (\Gen(N) \cup
\{0\})\} = M \cap \Inj_{\mathrm{fin}} (\Omega).$$ The desired
conclusion then follows from Corollary~\ref{inf.coimage.monoid}.

For the converse, suppose that $M_{\mathrm{fin}}$ has the form specified in the statement. First, we note that this set is a subsemigroup of $\Inj(\Omega)$. For, given any two elements $f, g \in M_{\mathrm{fin}}$, we have $$fg \in \{f \in \Inj(\Omega) : |\Omega \setminus (\Omega)f| \in N \setminus (\Gen(N) \cup \{0\})\},$$ by Lemma~\ref{coimage}. Further, by the same lemma and Proposition~\ref{fin.equiv.main}, $M_{\mathrm{fin}}$ is closed under multiplication by elements of $\Fin(\Omega)$, and, as always, $\Inj_{\infty}(\Omega)$ is closed under multiplication by elements of $\Fin(\Omega) \cup M_{\mathrm{fin}}$. Therefore $M$ is indeed a submonoid. Since $\Fin(\Omega)$, $B$, $$\{f \in \Inj(\Omega) : |\Omega \setminus (\Omega)f| \in N \setminus (\Gen(N) \cup \{0\})\},$$ and $\Inj_{\infty}(\Omega)$ are normal, it follows that $M$ is as well.
\end{proof}

\section{Monoids containing $\Alt(\Omega)$}

This section is devoted to submonoids $M \subseteq \Inj(\Omega)$
satisfying $M \cap \Sym(\Omega) = \Alt(\Omega).$ As in the previous
section, we shall first describe how composition with elements of
$\Alt(\Omega)$ affects the cycle decomposition of an arbitrary
element of $\Inj(\Omega).$

\begin{lemma}\label{even.fin}
Let $h \in \Fin(\Omega)$ be any element, and let $f \in \Inj(\Omega)$ be a map that satisfies either of the following conditions:
\begin{enumerate}
\item[$(${\rm i}$)$] $(f)\C_{\mathrm{open}} + (f)\C_{\mathrm{fwd}} \geq 2;$
\item[$(${\rm ii}$)$] $(f)\C_{\mathrm{open}} + (f)\C_{\mathrm{fwd}} \geq 1$ and $\, (f)\C_n = \aleph_0$ for some $n \in \Z_+$.
\end{enumerate}
Then there exists a map $g \in \Alt(\Omega)$ such that $fh$ and
$fg$ have equivalent cycle decompositions $($and hence so do $hf$
and $gf$$)$.
\end{lemma}

\begin{proof}
If $h \in \Alt(\Omega),$ then there is nothing to prove. Let us
therefore assume that $h \in \Fin(\Omega) \setminus \Alt(\Omega).$
Then for any transposition $a \in \Fin(\Omega),$ we have $ha \in
\Alt(\Omega).$ In both cases, we shall define $g = ha$, using an appropriate transposition $a.$

Now, assume that $f$ satisfies (i). Then, by Corollary~\ref{fin.equiv.1}, $fh$ must have at least two infinite cycles in its cycle decomposition. Let $\Sigma, \Gamma \subseteq \Omega$ be such (distinct) cycles, and let us pick $\sigma \in \Sigma$ and $\gamma \in \Gamma$ arbitrarily. Let $a
\in \Fin(\Omega)$ be the transposition that interchanges $\sigma$
and $\gamma,$ and fixes all other elements of $\Omega.$ Then $ha \in
\Alt(\Omega)$, and $fha$ has a cycle decomposition equivalent to that of $fh$ (by the same argument as in the proof of Lemma~\ref{fin.equiv.0.5}). 

Next, assume that $f$ satisfies (ii). Again, by Corollary~\ref{fin.equiv.1}, $fh$ must have at least one infinite cycle in its cycle decomposition and satisfy $(fh)\C_n = \aleph_0$. Thus, by Lemma~\ref{fin.equiv.2}, there exists a transposition $a \in \Fin(\Omega)$ such that $(fha)\C_n = (fh)\C_n + 1$ and $(fha)\C_m = (fh)\C_m$ for all $m \in \Z_+ \setminus \{n\}.$ Since $(fh)\C_n = \aleph_0$, it follows (by Corollary~\ref{fin.equiv.1}, once more) that $fha$ has a cycle decomposition equivalent to that of $fh$, as before.

The parenthetical statement follows from the fact that for any $f
\in \Inj(\Omega)$ and $h \in \Sym(\Omega),$ $hf$ and $fh =
h^{-1}(hf)h$ have equivalent cycle decompositions, by
Proposition~\ref{decomposition}.
\end{proof}

With the above lemma and the results of the previous section in
mind, to describe the effect of composing elements of $\Alt(\Omega)$ with
elements of $\Inj(\Omega) \setminus \Sym(\Omega)$ we only need to
consider maps having exactly one infinite cycle and finitely many $n$-cycles for each $n \in \Z_+$. The following equivalence relation will help us accomplish the task.

\begin{definition} \label{even.equiv.def}
Given any two maps $f,g \in \Inj(\Omega),$ let us write $f
\approx_{\mathrm{even}} g$ if $f \approx_{\mathrm{fin}} g$ and
$$\sum_{n \in \Z_+} ((f)\C_n - (g)\C_n)$$ is an even integer. $($Here $(f)\C_n - (g)\C_n$ is understood to be $\, 0$ whenever $\, (f)\C_n = (g)\C_n$, even if both cardinals are infinite.$)$

Further, we shall say that a subset $B \subseteq \Inj(\Omega)$ is $\, \approx_{\mathrm{even}}$-closed if for all $f, g \in \Inj(\Omega)$ such that $f \approx_{\mathrm{even}} g$, $f \in B$ if and only if $g \in B.$
\end{definition}

\begin{lemma}\label{even.equiv.rel}
The binary relation $\, \approx_{\mathrm{even}}$ on elements of $\,
\Inj(\Omega)$ is an equivalence relation.
\end{lemma}

\begin{proof}
It is clear that $\approx_{\mathrm{even}}$ reflexive and symmetric. Let us then suppose that $f \approx_{\mathrm{even}} g$ and $g \approx_{\mathrm{even}} h$ for some $f,g,h \in \Inj(\Omega)$, and show that $f \approx_{\mathrm{even}} h.$ Since $\approx_{\mathrm{fin}}$ is an
equivalence relation, we only need to show that the integer
$$\sum_{n \in \Z_+} ((f)\C_n - (h)\C_n)$$ is even. Let
$I \subseteq \Z_+$ be a finite set such that for all $n \in \Z_+
\setminus I,$ $(f)\C_n = (g)\C_n = (h)\C_n.$ (Such a set
exists because $f \approx_{\mathrm{fin}} g \approx_{\mathrm{fin}}
h.$) Computing modulo $2,$ we have $$0 \equiv \sum_{n \in \Z_+} ((g)\C_n - (h)\C_n) = \sum_{n \in I} ((g)\C_n - (h)\C_n) = \sum_{n \in I} (g)\C_n - \sum_{n \in I} (h)\C_n.$$ Hence $$\sum_{n \in \Z_+} ((f)\C_n -
(h)\C_n) =  \sum_{n \in I} (f)\C_n -  \sum_{n \in I}
(h)\C_n \equiv \sum_{n \in I} (f)\C_n -  \sum_{n \in I}
(g)\C_n$$ $$= \sum_{n \in \Z_+} ((f)\C_n - (g)\C_n) \equiv 0.$$
\end{proof}

We shall prove that for $f,g \in \Inj(\Omega)$, each having exactly one infinite cycle and finitely many $n$-cycles for each $n \in \Z_+$, $f \approx_{\mathrm{even}} g$ if and only if $f$ is conjugate to $h_1g h_2$ for some $h_1, h_2 \in \Alt(\Omega).$ The argument proceeds through three lemmas.

\begin{lemma}\label{even.alt.1}
Let $f,g \in \Inj(\Omega)$ be any two maps, each having at least one infinite cycle in its cycle decomposition. If $f \approx_{\mathrm{even}} g,$ then $f = hgh'h^{-1}$ for some $h \in \Sym(\Omega)$ and $h' \in \Alt(\Omega).$
\end{lemma}

\begin{proof}
Suppose that $f \approx_{\mathrm{even}} g$. Repeatedly applying Lemmas~\ref{fin.equiv.2} and~\ref{fin.equiv.3}, we can find a finite sequence $h_1, \dots, h_m \in \Fin(\Omega)$ of transpositions such that $gh_1 \dots h_m$ has a cycle decomposition equivalent to that of $f.$ Since $$\sum_{n \in \Z_+} ((f)\C_n - (g)\C_n)$$ is an even
integer, we can pick  $h_1, \dots, h_m$ so that $m$ is even as well, and hence $h' = h_1 \dots h_m \in
\Alt(\Omega).$ The statement then follows from Proposition~\ref{decomposition}.
\end{proof}

By Lemma~\ref{even.fin}, the converse of the above lemma is generally false. However, we shall prove (in Proposition~\ref{even.alt.2}) that it holds for maps having exactly one infinite cycle and finitely many $n$-cycles for each $n \in \Z_+$.

\begin{lemma}\label{even.alt.3}
Let $f \in \Inj(\Omega)$ be a map that satisfies $\, (f)\C_{\mathrm{open}} + (f)\C_{\mathrm{fwd}} = 1$ and $\, (f)\C_n < \aleph_0$ for all $n \in \Z_+$, and let $h \in \Fin(\Omega) \setminus \{1\}$ be a transposition. Then $$\sum_{n \in \Z_+} ((f)\C_n - (fh)\C_n) = \sum_{n \in \Z_+} ((f)\C_n - (hf)\C_n) \in \{-1,1\}.$$
\end{lemma}

\begin{proof}
As noted before, $fh$ and $hf$ must have equivalent cycle
decompositions, and hence the two sums above must always be equal.
Thus it suffices to prove that the former, which we shall denote
by $A$ from now on, is either $-1$ or $1$. Let $\Sigma \subseteq
\Omega$ be the infinite cycle in the cycle decomposition of $f.$ Then we can write $\Sigma = \{\alpha_i\}_{i \in I}$, where $I$ is either $\Z$ or $\Z_+$, and $(\alpha_i)f = \alpha_{i+1}$ for all $i \in I$. We shall consider a number of different cases.

Suppose that $h$ interchanges some $\alpha_i$ and $\alpha_{i+n}$
($n > 0$). Then $\{\alpha_i, \alpha_{i+1}, \dots, \alpha_{i+n-1}\}$ is a finite cycle and $\Sigma \setminus \{\alpha_i, \alpha_{i+1}, \dots, \alpha_{i+n-1}\}$ is an infinite cycle under $fh.$ Thus $f$ and $fh$ have equal numbers of all types of cycles, except $fh$ has one additional cycle of
cardinality $n$ (given that $(f)\C_n < \aleph_0$). Therefore $A = -1.$

Next, suppose that $\Gamma = \{\beta_0, \dots, \beta_{m-1}\}$ is a
finite cycle under $f$ (with $m > 1$), where $(\beta_i)f = \beta_{i+1} \ (\mathrm{mod} \ m)$, and that $h$ interchanges
$\beta_0$ and $\beta_j$ $(0 < j \leq m-1)$. Then $f$ and $fh$ have
the same cycles, except in place of $\{\beta_0, \dots, \beta_{m-1}\},$
$fh$ has the two cycles $\{\beta_0, \dots, \beta_{j-1}\}$ and
$\{\beta_j, \dots, \beta_{m-1}\}.$ Hence, compared to $f,$ $fh$ has one fewer cycle of cardinality $m,$ one more cycle of cardinality $j$, and one more cycle of cardinality $m-j$. Therefore $A = -1.$

Now, let $\Sigma$ and $\Gamma$ be as before (though now $m$ is
allowed to be $1$), and suppose that $h$ interchanges $\beta_0$
and some $\alpha_i.$ Then $\Sigma \cup \Gamma$ is a cycle under
$fh,$ but otherwise $fh$ has the same cycles as $f.$ Hence $fh$
has one fewer cycle of cardinality $m$ than $f,$ and therefore $A
= 1.$

Finally, suppose that $\Gamma = \{\beta_0, \dots, \beta_{m-1}\}$ and
$\Delta = \{\delta_0, \dots, \delta_{n-1}\}$ are distinct finite
cycles under $f$ (where $m,n \geq 1$, $(\beta_i)f = \beta_{i+1} \ (\mathrm{mod} \ m)$, and $(\delta_i)f = \delta_{i+1} \ (\mathrm{mod} \ n)$), and that $h$ interchanges $\beta_0$ and $\delta_0$. Then $\Gamma \cup \Delta$ is a cycle under $fh,$ but otherwise $fh$ has the same cycles as $f.$ Hence, compared to $f,$ $fh$ has one fewer cycle of cardinality $m,$ one
fewer cycle of cardinality $n,$ and one more cycle of cardinality
$m+n.$ Therefore $A = 1.$

In all cases the sum $A$ is either $1$ or $-1$, as claimed.
\end{proof}

\begin{lemma}
Let $f \in \Inj(\Omega)$ be a map that satisfies $\, (f)\C_{\mathrm{open}} + (f)\C_{\mathrm{fwd}} = 1$ and $\, (f)\C_n < \aleph_0$ for all $n \in \Z_+$, and let $h \in \Alt(\Omega)$ be any map. Then $fh \approx_{\mathrm{even}} f \approx_{\mathrm{even}} hf.$
\end{lemma}

\begin{proof}
We shall only prove the first equivalence. By Lemma~\ref{even.equiv.rel}, it is enough to show this in the case where $h = h_1h_2$ for some transpositions $h_1$ and $h_2.$ By Proposition~\ref{fin.equiv.main}, $f \approx_{\mathrm{fin}} fh.$ Let $I \subseteq \Z_+$ be a finite set such that for all $n \in \Z_+ \setminus I,$ $(f)\C_n = (fh_1)\C_n = (fh)\C_n.$
Modulo $2,$ we have $$\sum_{n \in I} ((f)\C_n - (fh_1)\C_n) \equiv 1 \equiv \sum_{n \in I} ((fh_1)\C_n - (fh_1h_2)\C_n),$$ by Lemma~\ref{even.alt.3}. Hence, $$\sum_{n \in \Z_+} ((f)\C_n - (fh)\C_n) = \sum_{n \in I} ((f)\C_n - (fh)\C_n) \equiv \sum_{n \in I} (f)\C_n - (\sum_{n \in I} (fh_1)\C_n - 1) \equiv 0.$$ Thus, $fh \approx_{\mathrm{even}} f$, by Definition~\ref{even.equiv.def}.
\end{proof}

Combining the last three lemmas, we obtain a description of $\approx_{\mathrm{even}}$ in terms of composition with members of $\Alt(\Omega)$, for elements of $\Inj(\Omega)$ that do not satisfy the hypotheses of Lemma~\ref{even.fin}.

\begin{proposition}\label{even.alt.2}
Let $f,g \in \Inj(\Omega)$ be any two maps, each having exactly one infinite cycle in its cycle decomposition, and satisfying $\, (f)\C_n, (g)\C_n < \aleph_0$ for all $n \in \Z_+$. Then $f \approx_{\mathrm{even}} g$ if and only if $f = hh_1gh_2h^{-1}$ for some $h \in \Sym(\Omega)$ and $h_1, h_2 \in \Alt(\Omega).$
\end{proposition}

\begin{proof}
The forward implication was proved in Lemma~\ref{even.alt.1}. For the converse, let us suppose that $f = hh_1gh_2h^{-1}$ for some $h \in
\Sym(\Omega)$ and $h_1, h_2 \in \Alt(\Omega).$ By
Proposition~\ref{decomposition}, it suffices to show that $h_1gh_2
\approx_{\mathrm{even}} g.$ But, this follows from the previous
lemma (and Lemma~\ref{even.equiv.rel}).
\end{proof}

We are, at last, in a position to describe the normal submonoids of
$\Inj(\Omega)$ that have $\Alt(\Omega)$ as the group of units.

\begin{proposition}\label{alt.classif}
Let $M \subseteq \Inj(\Omega)$ be any submonoid such that $M \cap
\Sym(\Omega) = \Alt(\Omega).$ Then $M$ is normal if and only if either $M = \Alt(\Omega) \cup M_{\mathrm{fin}}$ or $M = \Alt(\Omega) \cup M_{\mathrm{fin}} \cup \Inj_{\infty}(\Omega)$, where $$M_{\mathrm{fin}} = B \cup \{f \in \Inj(\Omega) : |\Omega \setminus (\Omega)f| \in N \setminus (\Gen(N) \cup \{0\})\},$$ for some additive submonoid $N \subseteq \N$ and subset $B \subseteq \Inj_{\mathrm{fin}}(\Omega)$ that satisfies the following conditions:
\begin{enumerate}
\item[$(${\rm i}$)$] $B_\N = \Gen(N);$
\item[$(${\rm ii}$)$] $\{f \in B : (f)\C_{\mathrm{open}} + (f)\C_{\mathrm{fwd}} \geq 2\}$ and $\, \{f \in B : \exists n \in \Z_+ \, ((f)\C_n = \aleph_0)\}$ are $\, \approx_{\mathrm{fin}}$-closed;
\item[$(${\rm iii}$)$] $\{f \in B : (f)\C_{\mathrm{open}} + (f)\C_{\mathrm{fwd}} = 1 \text{ and } \, \forall n \in \Z_+ \, ((f)\C_n < \aleph_0)\}$ is $\, \approx_{\mathrm{even}}$-closed.
\end{enumerate}
\end{proposition}

\begin{proof}
Again, this proof is very similar to those of Propositions~\ref{gp=fin} and~\ref{fin.classif}.

Suppose that $M$ is normal. Let $N = M_\N,$ and set $$B = \{f \in M : |\Omega \setminus (\Omega)f| \in \Gen(N)\}.$$ Then $B$ clearly satisfies (i); it satisfies (ii), by Proposition~\ref{fin.equiv.main} and Lemma~\ref{even.fin}; and it satisfies (iii), by Proposition~\ref{even.alt.2}. Further, if $h \in \Inj(\Omega)$ is any element such that $|\Omega \setminus (\Omega)h| \in M_\N \setminus (\Gen(M_\N) \cup \{0\}),$ then $h \in M,$ by Theorem~\ref{composition}. Hence, $$B \cup \{f \in \Inj(\Omega) : |\Omega \setminus (\Omega)f| \in N \setminus (\Gen(N) \cup \{0\})\} = M \cap \Inj_{\mathrm{fin}} (\Omega).$$ The desired conclusion then follows from Corollary~\ref{inf.coimage.monoid}.

For the converse, suppose that $M_{\mathrm{fin}}$ has the form specified in the statement. By Propositions~\ref{fin.equiv.main} and~\ref{even.alt.2}, $B$ is normal and is closed under multiplication by elements of $\Alt(\Omega)$. Since $\Alt(\Omega),$ $$\{f \in \Inj(\Omega) : |\Omega \setminus (\Omega)f| \in N \setminus (\Gen(N) \cup \{0\})\},$$ and $\Inj_{\infty}(\Omega)$ are also normal, it follows that $M$ is as well. By the usual argument, it is easy to see that $M$ must be a submonoid.
\end{proof}

\section{Main result} \label{main.section}

Putting together the remarks made in Section~\ref{gen.sect} with Propositions~\ref{sym.classif}, \ref{gp=fin}, \ref{fin.classif} and \ref{alt.classif}, we obtain a classification of all the normal submonoids of $\Inj(\Omega)$. Some of the conditions are phrased differently in the theorem below than in the aforemetioned propositions, in order to make the statement more self-contained.

\begin{theorem}
Let $M \subseteq \Inj(\Omega)$ be a normal submonoid. Then $M = M_{\mathrm{gp}} \cup M_{\mathrm{fin}} \cup M_{\infty},$ where
\begin{enumerate}
\item[$(1)$] $M_{\mathrm{gp}} \in \{\{1\}, \Alt(\Omega), \Fin(\Omega), \Sym(\Omega)\};$  
\item[$(2)$] $$M_{\mathrm{fin}} = B \cup \{f \in \Inj(\Omega) : |\Omega \setminus (\Omega)f| \in N \setminus (\Gen(N) \cup \{0\})\}$$ for some additive submonoid $N \subseteq \N$, with minimal generating set $\, \Gen(N)$, and normal subset $B \subseteq \Inj(\Omega)$ that satisfies
$$\{|\Omega \setminus (\Omega)f| : f \in B\} = \Gen(N);$$
\item[$(3)$] $M_{\infty} \in \{\emptyset, \Inj_{\infty}(\Omega)\}$.
\end{enumerate}
If $M_{\mathrm{gp}} \neq \{1\}$, then $B$ must satisfy additional hypotheses, as follows.

If $M_{\mathrm{gp}} = \Alt(\Omega),$ then
\begin{enumerate}
\item[$(${\rm i}$)$] the subset of $B$ consisting of maps having at least two infinite cycles or infinitely many cycles of a particular finite cardinality in their cycle decompositions is $\, \approx_{\mathrm{fin}}$-closed $($see Definition~\ref{fin.equiv.def} for the notation $\, \approx_{\mathrm{fin}}$$);$
\item[$(${\rm ii}$)$] the subset of $B$ consisting of maps having exactly one infinite cycle and finitely many cycles of each finite cardinality in their cycle decompositions is $\, \approx_{\mathrm{even}}$-closed $($see Definition~\ref{even.equiv.def} for the notation $\, \approx_{\mathrm{even}}$$)$.
\end{enumerate}

If $M_{\mathrm{gp}} = \Fin(\Omega),$ then $B$ is $\, \approx_{\mathrm{fin}}$-closed.

If $M_{\mathrm{gp}} = \Sym(\Omega),$ then $$B = \{f \in \Inj(\Omega) :
|\Omega \setminus (\Omega)f| \in \Gen(N)\}.$$

Conversely, if $M = M_{\mathrm{gp}} \cup M_{\mathrm{fin}} \cup M_{\infty},$ where $M_{\mathrm{gp}}$ satisfies $(1)$, $M_{\mathrm{fin}}$ satisfies $(2)$, $M_{\infty}$ satisfies $(3)$, and $B$ satisfies the appropriate conditions above $($depending on the form of $M_{\mathrm{gp}}$$)$, then $M$ is a normal submonoid of $\, \Inj(\Omega)$.
\end{theorem}

\vspace{.1in}

\noindent
Department of Mathematics \newline
Ben Gurion University \newline
Beer Sheva, 84105 \newline
Israel \newline

\noindent Email: {\tt mesyan@bgu.ac.il}

\end{document}